\documentclass[reqno]{amsart}

\usepackage{amssymb}
\usepackage{setspace}
\usepackage{mathrsfs}
\usepackage{graphicx}
\usepackage{enumerate}
\usepackage{color}
\usepackage[colorlinks=false, urlcolor=blue, final]{hyperref}
\usepackage{pdfsync}
\usepackage[normalem]{ulem}
\newtheorem{thm}{Theorem}[section]
\newtheorem{lem}[thm]{Lemma}

\newtheorem{prop}[thm]{Proposition}

\theoremstyle{definition}
\newtheorem{defn}[thm]{Definition}

\theoremstyle{remark}
\newtheorem{rmk}[thm]{Remark}
\newtheorem{ex}[thm]{Example}
\newtheorem{claim}[thm]{Claim}
\newtheorem*{comment}{Comment}

\allowdisplaybreaks

\newcommand{\R}{\mathbb{R}}
\newcommand{\N}{\mathbb{N}}
\newcommand{\Z}{\mathbb{Z}}
\newcommand{\C}{\mathbb{C}}
\newcommand{\T}{\mathbb{T}}
\newcommand{\inv}{^{-1}}

\DeclareMathOperator{\interior}{Int}
\DeclareMathOperator{\ev}{ev}

\DeclareMathOperator{\cspn}{\overline{span}}
\DeclareMathOperator{\Iso}{Iso}

\def\jbchange{\textcolor{red}}
\def\sarchange{\textcolor{blue}}
\def\nagychange{\textcolor[rgb]{0,0.5,0}}
\def\nagychangeII{\textcolor[rgb]{0.5,0,0.5}}

\newenvironment{nagy}{\color[rgb]{0,0.5,0}}{\color{black}}
\newenvironment{brown}{\color{red}}{\color{black}}
\newenvironment{sar}{\color{blue}}{\color{black}}

\renewcommand\sout[1]{}

\renewcommand\jbchange{}
\renewcommand\sarchange{}
\renewcommand\nagychange{}
\renewcommand\nagychangeII{}

\renewenvironment{brown}{}{}

\begin{document}

\title[A generalized Cuntz-Krieger uniqueness theorem]{A generalized Cuntz-Krieger uniqueness theorem for higher rank graphs}
\author{Jonathan H.~Brown}\address{Department of Mathematics, 138 Cardwell Hall, Kansas State university, Manhattan KS 66506, U.S.A., 785-532-0595}
\email{brownjh@math.ksu.edu}
\author{Gabriel Nagy}
\address{Department of Mathematics, 138 Cardwell Hall, Kansas State university, Manhattan KS 66506, U.S.A., 785-532-0592}
%\curraddr{}
\email{nagy@math.ksu.edu}
\author{Sarah Reznikoff}
\address{Department of Mathematics, 138 Cardwell Hall, Kansas State university, Manhattan KS 66506, U.S.A., 785-532-0575}
  %\curraddr{}
\email{sararez@math.ksu.edu}
\thanks{This research was partially supported by NSF grant
      DMS 1201564.}

\begin{abstract} We present a uniqueness theorem for k-graph C*-algebras that requires
neither an aperiodicity nor a gauge invariance assumption.  Specifically,
 we prove that for the injectivity of a representation of a k-graph
 C*-algebra, it is sufficient that the representation be injective on a
 distinguished abelian C*-subalgebra. A crucial part of the proof is the
 application of an abstract uniqueness theorem, which says that such a
 uniqueness property follows from the existence of a jointly faithful
 collection of states on the ambient C*-algebra, each of which is the
 unique extension of a state on the distinguished abelian C*-subalgebra.

 \end{abstract}

\maketitle
\section{Introduction}
A higher-rank graph of dimension $k$ or a {\em $k$-graph} is a generalization of a directed graph where paths ``look like'' convex subsets of the coordinate lattice in $\R^k$; directed graphs are canonically identified with $1$-graphs.  Kumjian and Pask introduced higher-rank graphs in 2000 \cite{KP00} and associated to a $k$-graph $\Lambda$ (row-finite and with no sources),  a $C^*$-algebra $C^*(\Lambda)$, which is universal for certain relations encoded in the $k$-graph.   Kumjian and Pask developed higher-rank graph graph algebras to simultaneously generalize graph algebras \cite{KPRR97} and the higher-rank Cuntz-Kreiger algebras constructed by Robertson and Steger \cite{RS99}.  Higher-rank graph algebras extend the class of $C^*$-algebras that can be studied by the combinatorial methods developed in the graph algebra literature.  For example, Pask, Raeburn, R\o rdam and A. Sims (\cite{PRRS06}) use higher-rank graphs to construct $A\T$ algebras which are simple but neither $AF$ nor purely infinite, and thus can not be realized as $1$-graph $C^*$-algebras \cite{KPR98}.  But the path structure of the higher-rank graph still reveals when the resulting algebra is simple  \cite{RS07} or purely-infinite \cite{Sim06} and determines  the (gauge-invariant) ideal structure \cite{KP00}  in much the same way it does for graph algebras.

In this paper we study representations of higher-rank graph algebras.  For a $k$-graph $\Lambda$ the universality of $C^*(\Lambda)$ states that a map from a higher-rank graph $\Lambda$ to a $C^*$-algebra $B$ that respects the Cuntz-Kreiger relations (see Definition~\ref{CK} below) gives a unique algebra homomorphism $C^*(\Lambda)\to B$.   This makes representations of graph algebras particularly easy to construct and shows $C^*(\Lambda)$ is generated by partial isometries of $\{S_\alpha \, |\, \alpha\in \Lambda\}$.  Kumjian and Pask \cite{KP00}  provide two uniqueness theorems that  guarantee a representation is injective: the \emph{gauge invariant uniqueness theorem} and the \emph{Cuntz-Krieger uniqueness theorem}.   The former requires that the representation transfer an action of $\T^k$ on the graph algebra to a well-defined action on the range space, the latter requires a fairly stringent condition (\nagychange{aperiodicity}) on the graph itself.   Once these conditions are satisfied then  both theorems say  representations $\pi: C^*(\Lambda)\to B$ that are injective on a certain canonical abelian subalgebra, $\mathscr{D}=C^*(\{S_\alpha^{} S_\alpha^* \, |\, \alpha\in \Lambda\})$ are injective.

As Szyma\'{n}ski points out \cite{szy02}, both the Cuntz-Krieger and gauge invariant u\-nique\-ness theorems are not applicable in certain situations; he then goes on to prove a more general version of these theorems in the context of $1$-graphs.  In this vein the second and third named authors \cite{NR12} prove a uniqueness theorem for $1$-graphs by relating the injectivity of a representation to injectivity of its restriction to a certain canonical abelian subalgebra. In this paper, we prove a uniqueness theorem (Theorem~\ref{thm: inj}) for higher-rank graphs, in which the aperiodicity condition is removed, thus going beyond the uniqueness theorems of Kumjian and Pask.  \jbchange{For  a row finite $k$-graph $\Lambda$ with no sources, we }identify an abelian subalgebra $\mathscr{M}$ of $C^*(\Lambda)$ such that, for a homomorphism $\pi: C^*(\Lambda)\to B$,
the injectivity of $\pi\big|_{\mathscr{M}}$ forces the injectivity of $\pi$. In general, the abelian subalgebra $\mathscr{M}$ used in Theorem~\ref{thm: inj} contains $\mathscr{D}$ and this containment is strict unless $\mathscr{D}$ is maximal abelian. The algebra $\mathscr{M}$ is generated by elements of the form $S_\alpha S_\beta^*$ with $(\alpha,\beta)\in \Lambda\times \Lambda$ \jbchange{cycline (See Definition~\ref{def: cycline})}.  Cycline pairs are related to the generalized cycles \jbchange{without entry} of
Evans and Sims \cite{ES}.

Theorem~\ref{thm: inj} extends the result of the second and third named authors
\cite[Theorem~3.13]{NR12}, and some steps are inspired by their proof, as well as by \cite{NR13}, where a framework for
``abstract'' uniqueness theorems is introduced.
As in the proof of the above mentioned result, we identify a dense subset of infinite paths $\mathfrak{T}$ (see Definition~\ref{def: reg path}) of $\Lambda$ and construct an injective representation of $C^*(\Lambda)$ on $C(\T^k)\otimes \nagychangeII{B(\ell^2(\mathfrak{T}))}$.  As cycles in higher-rank graphs are more complicated than those in $1$-graph algebras, identifying a suitable
$\mathfrak{T}$ is far more subtle in our situation; in particular it requires the use of the Baire Category Theorem.
Furthermore unlike the proof \cite[Theorem~3.13]{NR12}, the uniqueness theorem from \cite{NR13}, as well as of the
uniqueness theorems of Kumjian and Pask \cite{KP00},
 whose proofs required the construction of a (unique) conditional expectation from the algebra to the abelian subalgebra,
 our proof eliminates the need for conditional expectations. Instead, we analyze the states with unique extension property on $\mathscr{M}$ (Theorem~\ref{thm: state ext}), and generalize the uniqueness theorem from \cite{NR13}.  It would be useful to know whether  there exists a conditional expection of $C^*(\Lambda)$ onto $\mathscr{M}$ but we were unable to find a candidate.
%\footnote{J:  give outline when we have finalized structure}

%%%%%%%%%%%%%%%%%%%%%%%%%%%%%%%%%%%%%%%%%%%%%%%%%%%%%%%%%%%%%%%%%%%%%%%%%%%%%%%%%%%%%%%%%%%%%%%%%%%%%%%%%%%%%%%%%%%%%%%%%%%%%%%%%%%%%%%%%%%%%%%%%%%%%%%%%%%%%%%%%%%

\section{Preliminaries}

For any subset $D\subset X$ of a topological space $X$ we denote its closure by $\overline{D}$, its interior by $\interior(D)$ and its boundary by $\partial D$.

\subsection{Higher rank graphs}
We recall the basic definitions and terminology from \cite{KP00}. For any category $\mathscr{C}$ there exist two maps $r,s$  from the morphisms in $\mathscr{C}$ to its objects, so that for any morphism $c$, $r(c)$ is the range of the morphism and $s(c)$ is the source (or domain) of the morphism.  We identify the objects in all categories with their identity morphisms.

Let $k\in \N$.  A $k$-graph $\Lambda$ is a countable category along with a degree map $d_{\Lambda}: \Lambda\to \N^k$ that satisfies
\begin{enumerate}
\item $d_\Lambda(\mu\nu)=d_{\Lambda}(\mu)+d_{\Lambda}(\nu)$ for all $\mu,\nu\in \Lambda$ with $r(\nu)=s(\mu)$ and
\item  \emph{the unique factorization property:} For any $\lambda\in \Lambda$ such that $d_{\Lambda}(\lambda)=m+n$   there exists unique $\mu,\nu\in \Lambda$ with $\lambda=\mu\nu$ and $d_{\Lambda}(\mu)=m$ and $d_{\Lambda}(\nu)=n$.
\end{enumerate}
We denote $d_\Lambda$ by $d$ when the $k$-graph is clear from context.  We refer to objects (equivalently identity morphisms) in $\Lambda$ as \emph{vertices} and morphisms as \emph{paths}.  For $n\in \N^k$ define $\Lambda^n=d\inv\{n\}$; by unique factorization $\Lambda^0$ is the set of vertices in $\Lambda$.  For any $v\in \Lambda^0$ we denote $v\Lambda:=\{\lambda \, |\, r(\lambda)=v\}$. The sets $\Lambda v$, $v\Lambda^n$ and $\Lambda^n v$ are defined analogously.
A $k$-graph $\Lambda$ has \emph{no sources} if $v\Lambda^n$ is nonempty for all $v\in \Lambda^0$ and $n\in \N^k$;  $\Lambda$ is \emph{row-finite} if for all $v\in \Lambda^0$ and $n\in \N^k$, the set $v\Lambda^n$ is finite.

Let $\Omega_k:=\{(m,n)\in \N^k\times \N^k \, | \, m\leq n\}$.  (The order relation on $\N^k$ is the coordinate-wise order, that is, for $m=(m_1,\dots,m_k)$ and $n=(n_1,\dots,n_k)$, $m\leq n$ means
$m_j\leq n_j$, $\forall\,j=1,\dots,k$.)
Then $\Omega_k$ is a category with objects $\N^k$ and composition given by $(m,n)(n,r)=(n,r)$; it becomes a $k$-graph with degree map $d_\Omega(m,n)=n-m$.  The infinite path space $\Lambda^\infty$ of $\Lambda$ is the set of all degree-preserving covariant functors $x: \Omega_k\to \Lambda$.   By unique factorization infinite paths are completely determined by \jbchange{any sequence in $\N^k$ that is unbounded in every component direction.}

For $\mu\in \Lambda$, define
\[
Z(\mu):=\{x\in\Lambda^\infty \,  |\,  x(0, d(\mu))=\mu\}.
\]
The sets $\{Z(\mu)\}_{\mu\in \Lambda}$ define a basis for a totally disconnected second countable locally compact Hausdorff topology topology on $\Lambda^\infty$.

For $\alpha\in \Lambda$ and $x\in s(\alpha)\Lambda^\infty$, define $\alpha x$ to be the unique infinite path such that for any $n\geq d(\alpha)$, $(\alpha x)(0,n)=\alpha( x(0, n-d(\alpha)))$.

For any $p\in\N^k$, define \jbchange{\sout {the $p$-tail map}} $\sigma^p:\Lambda^\infty\to\Lambda^\infty$ by:
\[
\sigma^p(x)(m,n)=x(m+p,n+p),\quad\forall\,x\in\Lambda^\infty,\,(m,n)\in\Omega_k.
\]
Equivalently (by unique factorization), for $x\in\Lambda^\infty$,
$\sigma^p(x)\in\Lambda^\infty$ is the unique infinite path such that
$x(0,p)\sigma^p(x)=x$.

An infinite path $x\in \Lambda^\infty$ is \emph{periodic} if there exists $n\neq m\in \N^k$ such that $\sigma^n(x)=\sigma^m(x)$; $x$ is \emph{aperiodic} otherwise. \jbchange{Note that if there exists $\alpha\neq \beta$ and $y\in \Lambda^\infty$ such that $\alpha y=\beta y$ then $\alpha y$ is periodic.  A} $k$-graph $\Lambda$ is \emph{aperiodic} if for every $v\in \Lambda^0$, the set $v\Lambda^\infty$ contains an aperiodic path.

\begin{ex} Let $E$ be a directed graph.  If $e_1,\ldots e_n$ are edges in $E$ then $e_1e_2\cdots e_n$ is a path of length $n$ in $E$ if $r(e_i)=s(e_{i-1})$.  We denote the set of finite paths in $E$ by $E^*$.  When equipped with the degree map into $\N$ given by the length function, $E^*$ becomes a $1$-graph.  In a slight abuse of notation we denote this $1$-graph by $E$.
\end{ex}

\begin{ex}
\label{ex: fE} Suppose $E$ is a $1$-graph and $f:\N^k\to \N$ is a nonzero monoid homomorphism. Then the set $\Lambda=f^*E:=\{ (\mu, m): \mu\in E, \, m\in \N^k,\quad\text{and}\quad f(m)=d_E(\mu)\}$ with degree map $d(\mu,m)=m$ is a $k$-graph.
\end{ex}

\jbchange{For use later we show that for $\Lambda=f^* E$,  the space $\Lambda^\infty$ can be identified with $E^\infty$.}

\begin{prop}
\label{prop infty paths}
Let $E$ is a $1$-graph and $f:\N^k\to \N$ is a nonzero monoid homomorphism and $\Lambda=f^* E$.  Let $p_1: \Lambda \to E$ be the projection onto the first factor.  Then $p_1$ induces a homeomorphism from $\Lambda^\infty$  to $E^\infty$.
\end{prop}
First we need the following claim.

\begin{claim}
\label{clm: well def}
Let $x\in \Lambda^\infty$ and $m,n\in \N^k$ with $f(m)=f(n)$. Then $p_1 x(0, m)=p_1 x(0,n)$.
\end{claim}

\begin{proof}
$p_1 (x(0,m))p_1(x(m, m+n))=p_1(x(0,m+n))=p_1(x(0,n))p_1(x(n,n+m))$.  Now $d_E(p_1 (x(0,m)))=f(m)=f(n)=d_E(p_1(x(0,n)))$ so by unique factorization $p_1(x(0,m))=p_1(x(0,n))$ as desired.
\end{proof}

\begin{proof}[Proof of Proposition~\ref{prop infty paths}]
Since $f$ is nonzero there exists an $n\in \N^k$  with $f(n)\neq 0$. For $x\in \Lambda^\infty$ take $p_1 x$ to be the unique infinite path in $E^\infty$ characterized by $(p_1 x)(0, t(f(n)))=p_1(x(0, tn))$ for all $t\in \N$.  By Claim~\ref{clm: well def} this does not depend on the choice of $n$.    By a slight abuse of notation we denote the map $x\mapsto p_1 x$ by $p_1$.  Since $x(0,m)=(p_1(x(0,m)),m)$, if $p_1x=p_1 y$ then $x=y$: that is $p_1$ is injective.  For $z\mapsto((p,q)\mapsto (z(f(p), f(q)),q-p)$ is an inverse to $p_1$ so $p_1$ is surjective.  By Claim~\ref{clm: well def}  we also have $p_1(Z(\alpha))=Z(p_1(\alpha))$, so that $p_1$ is continuous and open and hence a homeomorphism.
\end{proof}

\subsection{\boldmath{$k$}-graph \boldmath{$C^*$}-algebras}

\nagychange{
\begin{defn}
\label{CK}
Let $\Lambda$ be a row-finite $k$-graph with no sources, and let $A$ be a $C^*$-algebra.
A Cuntz-Krieger
$\Lambda$-family in $A$ is a functor $\lambda\mapsto T_\lambda$ from $\Lambda$ into the set of partial isometries of $A$ such that:
\begin{enumerate}
\item[(CK1)]\label{CK1} $T_{\mu}^*T^{}_{\nu}=\delta_{\mu,\nu} T^{}_{s(\mu)}T_{s(\mu)}^*$,
for any $n\in \N^k$ and $\mu, \nu\in \Lambda^n$;  and
\item[(CK2)]\label{CK2} $T_{v}^*T^{}_v=\sum_{\lambda\in v\Lambda^n} T^{}_\lambda T_\lambda^*$,
for any $n\in \N^k$ and $v\in \Lambda^0$.
\end{enumerate}
\end{defn}}

\jbchange{
A quick computation taking $n=0$ in (CK2) and using functoriality shows $T_v$ is a projection (i.e. $T^2_v=T^*_v=T^{}_v$) for all $v\in \Lambda^0$.
Further, by taking $n=0$ in (CK1), it follows that that $T^{}_v$ and $T^{}_w$ are orthogonal for $v\neq w\in \Lambda^0$. Thus the relations presented here are equivalent to those given in \cite{KP00}.
}

The $C^*$-algebra of $\Lambda$, denoted $C^*(\Lambda)$, is the unique (up to isomorphism) $C^*$-alge\-bra generated by a universal Cuntz-Krieger $\Lambda$-family $S$.  That is if $T$ is a Cuntz-Krieger $\Lambda$-family in $A$ then
$S^{}_\lambda\mapsto T^{}_\lambda$ induces a unique $*$-homomorphism from $C^*(\Lambda)$ to $A$.

\nagychange{Using (CK1) and (CK2) it follows immediately that
$$C^*(\Lambda)=\overline{\text{span}}\{S^{}_\mu S^*_\nu\,:\,
\mu,\nu\in\Lambda,\,s(\mu)=s(\nu)\}.$$
The elements $S^{}_\mu S^*_\nu$ in the above list are referred to as the {\em standard generators}.
(The list does not contain the products $S^{}_\mu S^*_\nu$ with
$s(\mu)\neq s(\nu)$, because all such products are zero.)
For convenience for each $\lambda\in \Lambda$ we denote $P^{}_\lambda:= S^{}_\lambda S_\lambda^*$.
Observe that for  $v\in \Lambda^0$, $P^{}_v= S^{}_v$.}
\begin{brown}
Define
\[
\mathscr{D}:=C^*(\{P^{}_\mu \,  |\, \mu\in \Lambda\}).
\]
The subalgebra $\mathscr{D}$ is abelian with spectrum $\Lambda^\infty$: that is, $\mathscr{D}\cong C_0(\Lambda^\infty)$. Under this isomorphism $P_\mu$ is taken to the characteristic function on $Z(\mu)$.
\end{brown}
\begin{defn}
Consider the compact abelian group $\T^k$, which has $\mathbb{Z}^k$ as its character group, identified as follows.
For any $n=(n_1,\dots,n_k)\in\mathbb{Z}^k$, the associated character is the map $h_n:\T^k\to \T$, defined by
$h_n(z)=z_1^{n_1}\cdots z_k^{n_k}$, for all $z=(z_1,\dots,z_k)\in\T^k$.
\end{defn}

Since for any $z\in \T^k$, the correspondence $\Lambda\ni\lambda\mapsto h_{d(\lambda)}(z) S^{}_\lambda\in C^*(\Lambda)$ is another Cuntz-Krieger $\Lambda$-family in $C^*(\Lambda)$, it induces an automorphism of $C^*(\Lambda)$.  This gives an action of $\T^k$ on $C^*(\Lambda)$, which is referred to as the {\em gauge action}.

The \emph{gauge invariant uniqueness theorem} of \jbchange{Kumjian and Pask \cite[Theorem~3.4]{KP00} \sout{Raeburn, Sims and Yeend \cite[Theorem~4.1]{RSY03} }}%\footnote{J: Kumjian and Pask state the gauge invariant uniqueness theorem in the correct generality for us and we use the KP reference before. I think we should be consistent.}
states that, if $A$ is a C*-algebra with an action of $\T^k$,
and $\Phi: C^*(\Lambda)\to A$ is a $\T^k$-equivariant $*$-homomorphism such that
 $\Phi(P^{}_v)\neq 0$ for all $v\in \Lambda^0$, then $\Phi$ is injective.

%%%%%%%%%%%%%%%%%%%%%%%%%%%%%%%%%%%%%%%%%%%%%%%%%%%%%%%%%%%%%%%%%%%%%%%%%%%%%%%%%%%%%%%%%%%%%%%%%%%%%%%%%%%%%%%%%%%%%%%%%%%%%%%%%%%%%%%%%%%%%%%%%%%%%%%%%%%%%%%%%%%

\subsection{States and Representations}

For any $C^*$-algebra $A$, we denote its state space by $S(A)$,  the set of all pure states on $A$ by $P(A)$, \jbchange{and the set of quasi states by $Q(A)$: that is, $Q(A)=\{\phi \, | \, \phi \text{ positive}, \|\phi\| \leq 1\}$. Recall $Q(A)$ is a $w^*$-compact convex set.}

%If $X$ is a locally compact Hausdorff space, we denote by $C_0(X)$ the abelian $C^*$-algebra of continuous functions vanishing at infinity.
If $x$ is some point in $X$, we denote the evaluation map
$C_0(X)\ni f\mapsto f(x)\in\mathbb{C}$ by $\ev_x$, so that the correspondence $x\mapsto \ev_x$ establishes a homeomorphism
$X\simeq P(C_0(X))$, when we equip the range space with the $w^*$-topology. For each $x\in X$, the maximal ideal
$\ker(\ev_x)$ will be denoted by $C^x_0(X)$.

For future use we record here three easy technical results; the first two are well-known, but we
supply the proofs for the non-expert reader convenience.

\begin{prop}
\label{state eq}
Let $A$ be a C*-algebra and $\phi\in S(A)$.  If $b \geq 0$ is an element in $A$ with $\phi(b)=1=\|b\|$, then
$\phi(a)=\phi(ba)=\phi(ab)$, for all $a\in A$.

\begin{proof}  By extending $\phi$ to a state on the unitization $\tilde{A}$, we can assume $A$ is unital.
Since $0 \leq b \leq 1$, we have $0 \leq (1-b)^2 \leq 1-b$ and thus $0 \leq \phi((1-b)^2)\leq \phi(1-b)=0$.  Applying Cauchy-Schwartz, we have $|\phi(a(1-b))|^2 \leq \phi(aa^*)\phi((1-b)^2)$, and so $\phi(a)-\phi(ab)=\phi(a(1-b))=0$. \end{proof}
\end{prop}

\begin{prop}
\label{pure 1}
 Let $K$ be a compact Hausdorff space. If $\psi$ is a pure state on $C(K) \otimes A$, then there exists an $ x \in K$ and a pure state $\phi$  on $A$ such that~$\psi = \ev_x \otimes \phi$.
\end{prop}
\begin{proof}   Consider the map
\begin{brown}
\begin{align*}
\Phi:  K \times Q(A)&\to Q(C(K) \otimes A),\text{given by}\\
 (x,\phi)& \mapsto \ev_x\otimes\phi.
\end{align*}
\end{brown}
It is easy to see that $\Phi$ is $w^*$-continuous. \jbchange{Since $Q(A)$ is $w^*$-compact}, the set $F=\Phi(K \times Q(A))$ is $w^*$-compact in
$Q(C(K) \otimes A)$.

We claim that the $w^*$-closure of $\mathrm{conv}(F)$ is $Q(C(K) \otimes A)$.  Indeed,  if we had a strict inclusion  $\overline{\mathrm{conv}(F)}^{w^*} \subsetneq Q(C(K) \otimes A)$  we could find some $\phi_0 \in Q(C(K) \otimes A)$, a self-adjoint element $a \in C(K) \otimes A$ and an $\alpha \in \mathbb{R}$ such that $\phi(a) \leq \alpha < \phi_0(a)$ for all $\phi \in F$.  If we think of $a$ as a function $a:K \rightarrow A$, the first inequality tells us that $\phi(a(x)) \leq \alpha$ for all $x\in K$, $\phi \in Q(A)$, which forces $a \leq \alpha 1$, contradicting $\phi_0(a)>\alpha$.

 Using the claim and Milman's Theorem, it follows that all pure states $\psi$ on $C(K) \otimes A$ (which are automatically extreme points in $Q(C(K) \otimes A))$) are in $F$.  Thus, any such $\psi$ can be written as $\psi=\ev_x \otimes \phi$ for some $x \in K$ and $\phi \in Q(A)$.  Since $\psi$ is pure, $\phi$ must be pure as well.
 \end{proof}

\begin{prop}
\label{pure 2}
 \jbchange{Let  $K$ be a compact Hausdorff space, $D$ be an abelian subalgebra of a $C^*$-algebra $A$ and $\phi$ be a pure state on $D$ with unique extension to a (necessarily pure) state $\psi$ on $A$.    Consider the inclusion $\mathbb{C} \otimes D \subset C(K) \otimes A$ and the pure state $1 \otimes \phi$ on $\mathbb{C} \otimes D$.  If $\eta$ is a pure state on $C(K) \otimes A$ that extends $1 \otimes \phi$, then $\eta=\ev_x \otimes \psi$ for some $x \in K$.}
\end{prop}
 \begin{proof}
 By Proposition~\ref{pure 1}, $\eta=\ev_x \otimes \xi$ for some pure state $\xi$ on $A$.  Restricting $\eta$ to $\mathbb{C} \otimes D$ we see that $\psi(a)=\eta(1 \otimes a)=\xi(a)$ for all $a \in D$. Thus by uniqueness $\xi=\psi$. \end{proof}

%%%%%%%%%%%%%%%%%%%%%%%%%%%%%%%%%%%%%%%%%%%%%%%%%%%%%%%%%%%%%%%%%%%%%%%%%%%%%%%%%%%%%%%%%%%%%%%%%%%%%%%%%%%%%%%%%%%%%%%%%%%%%%%%%%%%%%%%%%%%%%%%%%%%%%%%%%%%%%%%%%%

\section{``Abstract'' uniqueness}

This section is devoted to a uniqueness result (Theorem \ref{prop: main state}), which provides a generalization of a certain ``abstract'' uniqueness result found in \cite{NR13}. Its proof relies on the following.

 \begin{lem}
 \label{lem: state rep}
  Let $A$ be a $C^*$-algebra. Assume $M\subset A$ is an abelian $C^*$-subalgebra, and $\phi\in P(M)$ is a pure state that extends uniquely to a (necessarily pure) state $\psi$ on $A$. Let $\pi_\psi:A\to B(L^2(A,\psi))$ denote the
	GNS representation associated with $\psi$.
	\begin{enumerate}
	\item \label{weak state} If $\rho: A \rightarrow B(H)$ is a $*$-representation such that $\ker \rho\big|_M \subseteq \text{\rm ker}\,\phi$,  then there exists a net of vectors $(\xi_{\lambda})_{\lambda}$ in $H$, $\|\xi_\lambda \|=1$ such that
\begin{equation}
\psi(a)=\lim_\lambda \langle \rho(a)\xi_{\lambda} |\xi_{\lambda}\rangle,\,\,\forall\,a\in A.
\label{state rep a}
\end{equation}
 In particular, we have the inclusion:
\begin{equation}
\ker\rho\subset\ker\pi_\psi.
\label{state rep a ker}
\end{equation}
	\item\label{it ab} If $M$ is contained in the center of $A$, i.e. $ab=ba$, $\forall\,a\in A,\,b\in M$, then:
\begin{equation}
\pi_\psi(a)=\psi(a)I,\,\,\,\forall\,a\in A,
\label{state rep b}
\end{equation}
	so in particular $\psi$ is a $*$-homomorphism.
	\end{enumerate}
\end{lem}

 \begin{proof}
Identify $M=C_0(X)$ and $\phi=\ev_x$, for some $x\in X$, so $\text{ker}\,\phi=C^x_0(X)$.

\eqref{weak state}
 Fix $f \in C_0(X)$ with $0 \leq f \leq 1$ and $f(x)=1$.  Since $\|(f\mod C_0^{x}(X)) \|=\|f\| = 1$,
it follows that
\[
\|f\| \geq \|(f \mod \ker \rho) \| = \|(f\mod C_0^{x}(X))\|=1,
\]
so $\|\rho(f)\|=1$.

 In particular, for every $\epsilon \in (0,1)$, there is some vector $v=v_{(f,\epsilon)} \in H$ with $\|v_{(f,\epsilon)}\|=1$ and $1-\epsilon^2 \leq \langle \rho(f)v | v \rangle \leq 1$.

 Consider now $(v_{(f,\epsilon)})$ as a net, where $f$ decreases to $\delta_{x}$ and $\epsilon$ decreases to $0$.  Since we have $\langle(I-\rho(f))v_{(f,\epsilon)} | v_{(f,\epsilon)}\rangle \leq \epsilon^2$, we also have $\|(I-\rho(f))v_{(f,\epsilon)}\| < \epsilon$.  By noting that $\displaystyle \lim_{(f,\epsilon)} \|fg-g(x)f\|=0$  for all $g \in C_{0}(X)$, it follows immediately that\[ \lim_{(f,\epsilon)} \langle \rho(g)v_{(f,\epsilon)} | v_{(f, \epsilon)} \rangle = g(x),
 \,\,\forall\,
g \in C_{0}(X).
 \]

Consider the states $\omega_{(f, \epsilon)} \in S(A)$ given by $\omega_{(f, \epsilon)}(a)=\langle \rho(a)v_{(f, \epsilon)} | v_{(f, \epsilon)}\rangle$ and extract a $w^*$-convergent subnet $\omega_{(f_\lambda,\epsilon_\lambda)} \to \omega \in Q(A)$.  In other words, if we let $\xi_{\lambda}=v_{(f_\lambda, \epsilon_\lambda)}$,
we have $\omega(a)=\lim_\lambda\langle\rho(a)\xi_\lambda|\xi_\lambda\rangle$. Since $\omega(g)=g(x)=\ev_{x}(g)=\phi(g)$ for all $g \in C_{0}(X)$, it follows that $\omega=\psi$, so \eqref{state rep a} holds.

To prove the inclusion \eqref{state rep a ker},
let $\eta\in L^2(A,\psi)$ denote the standard cyclic vector for $\pi_\psi$, so that $\pi_\psi(A)\eta$ is dense in $L^2(A,\psi)$, and
\begin{equation}
\langle\pi_\psi(a)\pi_\psi(a^{}_1)\eta|\pi_\psi(a^{}_2)\eta\rangle=\psi(a^*_2aa^{}_1),
\,\,\forall\,a,a^{}_1,a^{}_2\in A.
\label{GNS eta}
\end{equation}
Observe now that, if we start with some element $a\in\ker\rho$, then by \eqref{state rep a} we immediately get
$$\psi(a^*_2aa^{}_1)=\lim_\lambda\langle\rho(a^*_2aa^{}_1)\xi_\lambda|\xi_\lambda\rangle=0,
\,\,\forall\,a^{}_1,a^{}_2\in A,$$
which by \eqref{GNS eta} and the density of $\pi_\psi(A)\eta$ in $L^2(A,\psi)$ forces $\pi_\psi(a)=0$.

\eqref{it ab} Assume now that $M$ is contained in the center of $A$. Observe first that, for any $a=a^*\in A$, the subalgebra
$M(a)=C^*(\{a\}\cup M)\subset A$ is abelian, and by the unique extension, $\phi$ extends uniquely to a pure state on
$M(a)$, namely $\psi\big|_{M(a)}$; in particular it follows that $\psi$ is multiplicative on $M(a)$, so
\begin{equation}
\psi(ab)=\psi(ba)=\psi(a)\phi(b),\,\,\forall\,b\in M.
\label{psi-mult}
\end{equation}
Since \eqref{psi-mult} holds for all self-adjoint elements, it also holds for {\em all\/} $a\in A$.

Note that, for $b\in M$, using \eqref{psi-mult} and \eqref{GNS eta}, we get
\[
\langle\pi_\psi(b)\pi_\psi(a^{}_1)\eta|\pi_\psi(a^{}_2)\eta\rangle=
\phi(b)\langle\pi_\psi(a^{}_1)\eta|\pi_\psi(a^{}_2)\eta\rangle,\,\,\forall\,a^{}_1,a^{}_2\in A,
\]
so by the density of $\pi_\psi(A)\eta$ in $L^2(A,\psi)$, it follows that
\begin{equation}
\pi_\psi(b)=\phi(b)I,\,\,\forall\,b\in M;
\label{lem b}\end{equation}
in other words, \eqref{state rep b} holds, if $a\in M$.

To finish observe that it suffices to show $\pi_\psi(a)$ is a scalar multiple of the identity for all $a\in A$. (For if $\pi_\psi(a)=\lambda_aI$ then $a\mapsto \lambda_a$ is a state on $A$ which extends $\phi$ by \eqref{lem b}. So by uniqueness $\psi(a)=\lambda_a$ for all $a\in A$.)    In fact it suffices to show for $a_0\in A$ positive. \jbchange{It is enough to show the numerical range, $\{\langle\pi_\psi(a_0)\eta|\eta\rangle|\|\eta\|=1\}$, of $\pi_\psi(a_0)$ is a singleton.  By contradiction assume there exist two vectors $\eta_1,\eta_2\in L^2(A,\psi)$, with $\|\eta_1\|=\|\eta_2\|=1$, such that $\langle\pi_\psi(a_0)\eta_1|\eta_1\rangle\neq\langle\pi_\psi(a_0)\eta_2|\eta_2\rangle.$ But then by \eqref{lem b}, the functionals $\psi_i$ given by $\psi_i(\cdot)=\langle\pi_\psi(\cdot)\eta_i|\eta_i\rangle$ for $i=1,2$, are  two  different states on $A$ extending $\phi$. }
\end{proof}

\begin{defn}
\label{defn: jointly faithful}
We say a collection $W=\{\rho_i:A\to B_i\}_{i\in I}$ of positive linear maps on a $C^*$-algebra $A$ (the $B_i$'s are also assumed to be $C^*$-algebras) is \emph{jointly faithful} if whenever $a\in A$  is such that $\rho_i(a^*a)=0$, for all
$i\in I$, it follows that $a=0$.
\end{defn}

Using this terminology, we have the following ``abstract'' uniqueness result. (In the applications we have in mind in this article, only part B of the Theorem will be used.)

\begin{thm}
\label{prop: main state}
Let $A$ be a $C^*$-algebra, and let $M\subset A$ be an abelian $C^*$-subalgebra. Suppose $F$ is a nonempty subset of $P(M)$ is such that
every $\phi\in F$ extends uniquely to a state $\phi^A$ on $A$ (which is necessarily pure).

\begin{enumerate}[{\rm I.}]
\item\label{GNS faithful} If the family of GNS representations $\{\pi_{\phi^A}\}_{\phi\in F}$ is faithful on $A$, then the following uniqueness statement is true.
\begin{enumerate}
\item[{\rm (i)}] A $*$-homomorphism $\rho:A\to B$ is injective, if and only if its restriction $\rho\big|_M$ is injective.
\end{enumerate}

\item\label{joint faithful} If the family $\{\phi^A\}_{\phi\in F}$
 is jointly faithful on $A$, then, in addition to statement {\rm (i)} above, the following are also true.
\begin{enumerate}
\item[{\rm (ii)}] The commutant of $M$ in $A$, that is, the set
$$M'=\{a\in A\,|\,ab=ba,\,\,\forall\,b\in M\},$$
is a maximal abelian subalgebra (MASA) in $A$.
\item[{\rm (iii)}] For any intermediary abelian $C^*$-subalgebra $M\subset N\subset A$, the set $\{\phi^A\big|_{N}\}_{\phi\in F}$ is $w^*$-dense in $P(N)$.
\end{enumerate}
\end{enumerate}
\end{thm}

\begin{proof}
\eqref{GNS faithful} If $\rho\big|_M$ is injective, then for every $\phi\in F$, we have the inclusions
$\{0\}=\ker(\rho\big|_M)\subset\ker(\phi)$. Thus
by Lemma \ref{lem: state rep} we get
\begin{equation}
\ker(\rho)\subset\bigcap_{\phi\in F}\ker(\pi_{\phi^A})=\{0\}.
\label{thm; main state a}
\end{equation}

\eqref{joint faithful}  \jbchange{ First note that Hypothesis~\eqref{joint faithful} implies hypothesis~\eqref{GNS faithful} and hence statement (i). Indeed, if $\psi$ is a state on $A$  and $\pi_\psi$ is the associated GNS representation then $\pi_\psi(a)=0$ implies $\psi(a^*a)=0$.  So if $\{\phi^A\}_{\phi\in F}$ is jointly faithful on $A$  then $\{\pi_{\phi^A}\}_{\phi\in F}$ is too. }

To show (ii), note that  the unique extension property and joint faithfullness of state extensions pass to subalgebras containing $M$. Hence we can assume $A=M'$. (The reader is cautioned that this reduction  cannot be performed under hypothesis \eqref{GNS faithful}. In particular, statements (ii) and (iii) do not follow if \emph{only} hypothesis \eqref{GNS faithful} is assumed.) Under this assumption we show $A$ is abelian. Consider the direct sum representation
\[
\pi=\bigoplus_{\phi\in F}\pi_{\phi^{A}}:A\to
B\big(\bigoplus_{\phi\in F}L^2(A,\phi^{A})\big).
\]
Since  the family $\{\pi_{\phi^{A}}\}_{\phi\in F}$ is jointly faithful, \eqref{GNS faithful} shows $\pi$ is injective.  Now $A=M'$ so Lemma~\ref{lem: state rep} \eqref{it ab} shows that
\[
\pi(a)=\bigoplus_{\phi\in F}\phi^{A}(a)I_{L^2(A,\phi^{A})},\,\,\forall\,a\in A.
\]
 So $\pi(A)$ is abelian and since $\pi$ is injective, $A$ is abelian too.

To prove (iii),  as in part (ii), we can assume with out loss of generality that  $A=N$  is abelian. Then $A=C_0(X)$ for some locally compact Hausdorff space $X$.   By Lemma~\ref{lem: state rep}  \eqref{it ab} the set $S=\{\phi^A\}_{\phi\in F}$ consists of $*$-homomorphisms so we can identify $S$ with a subset of $X$.  By the hypothesis \eqref{joint faithful}, $\{\text{ev}_s\}_{s\in S}$ is faithful on $C_0(X)$, so $S$ must be dense in $X=P(A)$.
\end{proof}

\begin{comment}
In \cite{NR13} the second and third named authors introduced a certain class of abelian $C^*$-subalgebras, termed
{\em pseudo-diagonals}, for which the uniqueness property (i) does hold. These are precisely those abelian subalgebras $M\subset A$ which \jbchange{possess a conditional expectation $E:A\to M$} in addition to the main hypothesis and hypothesis \eqref{joint faithful} in the preceding theorem. As is turns out (see \cite{NR13}), if $M$ is a pseudo-diagonal in $A$, then $M$ must be a MASA in $A$, that is, $M'=M$. However, if the conditional expectation assumption is removed, then $M\subset M'$ may be a strict inclusion. (For example, by identifying the endpoints, one has an inclusion $M=C(\T)\subset C[0,1]=A$, which does satisfy the required hypothesis; since
$A$ is abelian, $M$ cannot be a MASA in $A$.) Nevertheless, by property (iii) applied to $N=M'$, we see that
$M$ is ``almost a MASA,'' at least from a spectral point of view.
Note also that, if
$M\subset A$ satisfies the main hypothesis and hypothesis ~\ref{joint faithful}, then so will $M'\subset A$, thus
$M'$ has a ``better chance'' of being a pseudo-diagonal. Presently we do not know whether this is the case in general.
However, if the ambient $C^*$-algebra $A$ is ``large enough,'' (e.g. a von Neumann algebra), then averaging over unitaries in $M$ will produce a conditional expectation of $A$ onto $M'$. More generally, if one has an injective
$*$-representation $\rho:A\to B(H)$ (for which the injectivity of $\rho\big|_M$ suffices), such that
$\rho(M')\subset B(H)$ is a von Neumann subalgebra (abelian, thus injective), then there exists a conditional expectation of $B(H)$ onto $\rho(M')$, which will clearly produce a conditional expectation of $A$ onto $M'$.
\end{comment}

%%%%%%%%%%%%%%%%%%%%%%%%%%%%%%%%%%%%%%%%%%%%%%%%%%%%%%%%%%%%%%%%%%%%%%%%%%%%%%%%%%%%%%%%%%%%%%%%%%%%%%%%%%%%%%%%%%%%%%%%%%%%%%%%%%%%%%%%%%%%%%%%%%%%%%%%%%%%%%%%%%%

\section{The cycline subalgebra}

In this section we study a special class of standard generators for $C^*(\Lambda)$, which will are used in the construction a distinguished $C^*$-subalgebra $\mathscr{M}\subset C^*(\Lambda)$.

As in Subsection 2.2, for every $\alpha\in \Lambda$ we denote the projection $S^{}_\alpha S_\alpha^*\in
C^*(\Lambda)$ by $P^{}_\alpha$.

Observe that $P_{\alpha}^{}=P_{\beta}^{}$ if and only if $Z(\alpha)=Z(\beta)$ and that
\[
P_{\alpha\gamma}^{}S_\alpha^{} S_\beta^* P^{}_{\beta\eta}=S_{\alpha\gamma}^{} S_\gamma^* S^{}_\eta S_{\beta\eta}^*.
\]
In particular if $d(\gamma)=d(\eta)$ then
\begin{equation}
\label{proj form}
P^{}_{\alpha\gamma}S^{}_\alpha S_\beta^* P^{}_{\beta\eta}=\begin{cases} S^{}_{\alpha\gamma}S_{\beta\gamma}^* & \text{if~} \gamma=\eta\\ 0 &\text{otherwise}
\end{cases}
\end{equation}

\begin{prop}
\label{prop: supernormal}
For $(\alpha,\beta)\in \Lambda\times \Lambda$ with $s(\alpha)=s(\beta)$, the following are equivalent:
\begin{enumerate}
\item \label{it: supernormal 1} $P^{}_{\alpha\gamma}=P^{}_{\beta\gamma}$ for all $\gamma\in s(\alpha)\Lambda$.
\item \label{it: supernormal 2} $S^{}_\alpha S_\beta^*$ is normal and commutes with $\mathscr{D}=C^*(\{P_\mu \,  | \, \mu\in \Lambda\})$.
\jbchange{\item \label{it: supernormal 3} $\alpha y=\beta y$ for all $y\in s(\alpha) \Lambda^\infty$.}
\end{enumerate}
\end{prop}

\begin{proof}
 $\eqref{it: supernormal 1}\Rightarrow \eqref{it: supernormal 2}$ Consider $\mu\in \Lambda$, and we prove that
$S^{}_\alpha S_\beta^*$ commutes with $P^{}_\mu$.  Since $P^{}_\mu=\sum_{\eta\in s(\mu)\Lambda^n} P^{}_{\mu\eta}$, without loss of generality we can assume $d(\mu)> d(\alpha)+d(\beta)$. \jbchange{There are two cases.  Case $1$: suppose $\mu=\alpha\gamma$ or $\beta\gamma$ for some $\gamma$.  By part~\eqref{it: supernormal 1} $P_{\alpha\gamma}=P_{\beta\gamma}$, so}
 \[
 P^{}_{\mu} S^{}_{\alpha} S_{\beta}^*=P_{\alpha \gamma}S^{}_{\alpha}S^{}_{\beta} =S^{}_{\alpha\gamma}S_{\beta\gamma}^*=S^{}_{\alpha} S_{\beta}^* P^{}_{\beta\gamma}=S^{}_{\alpha}S_{\beta}^* P^{}_{\mu}.
 \]

For case $2$, assume $\mu$ is neither $ \beta\gamma$ nor $\alpha\gamma$ for any $\gamma\in \Lambda$. Then $P^{}_\mu S^{}_\alpha S_\beta^*=0=S^{}_\alpha S_\beta^* P^{}_\mu$.
 Thus   $S^{}_\alpha S_\beta^*$ commutes with $\mathscr{D}$.  Further  $(S^{}_\alpha S_\beta^*)^*=S^{}_\beta S_\alpha^*$, so
 \[
  S^{}_\alpha S_\beta^*(S^{}_\alpha S_\beta^*)^*=P^{}_\alpha=P^{}_\beta=(S^{}_\alpha S_\beta^*)^*S^{}_\alpha S_\beta^*.
  \]
  Therefore $S^{}_\alpha S_\beta^*$ is normal, giving \eqref{it: supernormal 2}.

$\eqref{it: supernormal 2}\Rightarrow \eqref{it: supernormal 1}$  Let $\gamma\in s(\alpha) \Lambda$.   Since $S^{}_{\alpha}S_{\beta}^*$ is normal and commutes with  the projection $P^{}_{\alpha\gamma}$  it \nagychange{follows that} $P^{}_{\alpha\gamma} S^{}_\alpha S_\beta^*=S^{}_{\alpha\gamma}S_{\beta\gamma}^*$ is normal too.  Now
\[
P^{}_{\alpha\gamma}=S^{}_{\alpha\gamma}S_{\beta\gamma}^*(S^{}_{\alpha\gamma}S_{\beta\gamma}^*)^*=(S^{}_{\alpha\gamma}S_{\beta\gamma}^*)^*S^{}_{\alpha\gamma}S_{\beta\gamma}^*=P^{}_{\beta\gamma}
\]
giving \eqref{it: supernormal 1}.

\begin{brown}
$\eqref{it: supernormal 1}\Rightarrow\eqref{it: supernormal 3}$. Pick $y\in s(\alpha)\Lambda^\infty$.  By \eqref{it: supernormal 1}, $P_{\alpha (y(0,n))}=P_{\beta (y(0,n))}$ for all $n$: thus $Z(\alpha(y(0,n)))=Z(\beta(y(0,n)))$ for all $n$.  Therefore
\[
\{\alpha y\}=\bigcap_{n\in Z^k}Z(\alpha(y(0,n))=\bigcap_{n\in Z^k}Z(\beta(y(0,n)))=\{\beta y\}
\]
and so $\alpha y=\beta y$ giving \eqref{it: supernormal 3}.

$\eqref{it: supernormal 3}\Rightarrow\eqref{it: supernormal 1}$. Let $\gamma\in s(\alpha)\Lambda$ be given. Since $\gamma z\in s(\alpha) \Lambda^\infty$ for all $z\in s(\gamma)\Lambda^\infty$, by \eqref{it: supernormal 3}, $\alpha\gamma z=\beta \gamma z$ for all $z\in s(\gamma)\Lambda^\infty$.  Thus $Z(\alpha\gamma)=Z(\beta\gamma)$ and so $P_{\alpha\gamma}=P_{\beta\gamma}$ giving \eqref{it: supernormal 1}.
\end{brown}
\end{proof}

\begin{brown}
\begin{rmk}
In \cite{CKSS}, Carlsen-Kang-Shotwell and Sims say $\alpha,\beta\in \Lambda$ are equivalent if they satisfy \eqref{it: supernormal 3} in Proposition~\ref{prop: supernormal}.  They observe that this is indeed an equivalence relation on paths.
\end{rmk}
\end{brown}

\begin{defn}
\label{def: cycline}
\jbchange{A \emph{cycline pair} is a pair $(\alpha,\beta)\in\Lambda\times\Lambda$ with $s(\alpha)=s(\beta)$ satisfying the equivalent conditions of Proposition~\ref{prop: supernormal}. }
\end{defn}

Observe that $(\alpha,\beta)$ is cycline if and only if $(\beta, \alpha)$ is.  In this case, call the element $S^{}_{\alpha}S_{\beta}^*$ a \emph{cycline generator}. The $C^*$-subalgebra
\[
\mathscr{M}=C^*\{S^{}_\alpha S_\beta^*\,|\,(\alpha,\beta)\text{ cycline}\}\subset C^*(\Lambda)
\]
will be referred to as the {\em cycline subalgebra of $C^*(\Lambda)$.}
If we denote by $\mathscr{D}'$ the commutant of $\mathscr{D}$ in $C^*(\Lambda)$, that is
\[
\mathscr{D}'=\{a\in C^*(\Lambda)\,|\,aP^{}_\mu=P^{}_\mu a,\,\,\forall\,\mu\in\Lambda\},
\]
then by construction we have the inclusions $\mathscr{D}\subset\mathscr{M}\subset\mathscr{D}'$.

\sout{I moved Remark 6.3 here:}
\begin{rmk}
\label{rmk: ap}
\nagychangeII{
If $\Lambda$ is aperiodic, then $\mathscr{D}$ is a MASA. Therefore $\nagychangeII{\mathscr{D}'}=\mathscr{D}$ and hence $\mathscr{D}=\mathscr{M}$.}
\end{rmk}

\begin{rmk}
\label{rmk: Mab}
\nagychangeII{
By a direct computation on the generators, it can be shown that $\mathscr{M}$ is abelian.
Alternatively, this property also follows from the more general statement that $\mathscr{D}'$ is abelian, which will be proved in Section 7. (See Remark~\ref{rmk: ab} and Proposition~\ref{prop: com}.)}
\end{rmk}

\begin{rmk}
\begin{brown}
The subalgebra $\mathscr{M}$ is equal to $\cspn\{S^{}_\alpha S_\beta^*\,|\,(\alpha,\beta)\text{ cycline}\}$.  Indeed, for   $(\alpha,\beta), (\mu,\nu)$ then there exists an $n\in \N^k$ with

\begin{equation}
\label{eq spn}
S^{}_\alpha S_\beta^* S^{}_\mu S^{*}_\nu=\sum_{\substack{\gamma,\eta\in \Lambda\\  \beta\gamma=\mu\eta \\ d(\beta\gamma)=n}} S^{}_{\alpha\gamma}S_{\nu\eta}^*.
\end{equation}
If  $(\alpha,\beta), (\mu,\nu)$ are cycline and $\lambda\in s(\gamma)\Lambda$  then
\[
P_{\nu\eta\lambda}= P_{\mu\eta\lambda}=P_{\beta\gamma\lambda}=P_{\alpha\gamma\lambda}.
\]
The first and last equalities hold because $(\mu,\nu)$ and $ (\alpha,\beta)$  are cycline and the middle equality holds because $\mu\eta\lambda=\beta\gamma\lambda$.  Thus for  $(\alpha,\beta), (\mu,\nu)$  cycline each term on the right-hand side of \eqref{eq spn} is given by a cycline pair.
\end{brown}
\end{rmk}

\begin{prop}
\label{prop: no aperiodic}
\jbchange{Let $\Lambda$ be a $k$-graph, and $(\alpha,\beta)\in \Lambda\times \Lambda$   such that $s(\alpha)=s(\beta)$.}
\begin{enumerate}
\item\label{noa 1}  If $\alpha\neq \beta$ and there exists an aperiodic path $x\in s(\alpha)\Lambda^\infty$, then $(\alpha,\beta)$ is not a cycline pair.
\item\label{noa 2}  If $\Lambda$ is aperiodic then $(\alpha,\beta)$ is a cycline pair if and only if $\alpha=\beta$.
\end{enumerate}
\end{prop}

\begin{proof}
For \eqref{noa 1}, let $(\alpha,\beta)\in \Sigma$ with $\alpha\neq \beta$ and suppose $x\in s(\alpha)\Lambda^\infty$ is aperiodic.  Then both $\alpha x$ and $\beta x$ are aperiodic, so $\alpha\neq \beta$ implies $\alpha x\neq \beta x$.  Thus for all  $n\in \N^k$ sufficiently large $(\alpha x)(0,n)\neq (\beta x)(0,n)$.  Pick $n$ large and  $\gamma=x(0,n)$.
Then $P^{}_{\alpha\gamma} S^{}_{(\beta x)(0,n)}=S^{}_{\alpha\gamma} S_{\alpha\gamma}^*S^{}_{(\beta x)(0,n)}=S^{}_{\alpha\gamma}S_{(\alpha x)(n,d(\alpha\gamma))}^* S_{(\alpha x)(0,n)}^*S^{}_{(\beta x)(0,n)}=0$.  But $P^{}_{\beta\gamma} S^{}_{(\beta x)(0,n)}=S^{}_{\beta\gamma} S_{(\beta x)(n, d(\beta\gamma))}^*S_{(\beta x)(0,n)}^*S^{}_{(\beta x)(0,n)}=S^{}_{\beta\gamma} S_{(\beta x)(n, d(\beta\gamma))}^*\neq 0$.  Thus $P^{}_{\alpha\gamma} \neq P^{}_{\beta\gamma}$ and so $(\alpha,\beta)$ is not cycline.

%Then $P^{}_{\alpha\gamma} S^{}_{(\beta x)(0,n)}=S^{}_{\alpha\gamma} S_{\alpha\gamma}^*S^{}_{(\beta x)(0,n)}=S^{}_{\alpha\gamma}S_{(\alpha\gamma)(n,d(\alpha\beta))}^* S_{(\alpha x)(0,n)}^*S^{}_{(\beta x)(0,n)}=0$.  But $P^{}_{\beta\gamma} S^{}_{(\beta x(0,n))}=S^{}_{\beta\gamma} S_{(\beta x)(n, d(\beta\gamma))}^*S_{(\beta x)(0,n)}^*S^{}_{(\beta x)(0,n)}=S^{}_{\beta\gamma} S_{\beta x)(n, d(\beta\gamma))}^*\neq 0$.  Thus $P^{}_{\alpha\gamma} \neq P^{}_{\beta\gamma}$ and so $(\alpha,\beta)$ is not cycline.

Now  \eqref{noa 2} follows from \eqref{noa 1} by the definition of aperiodicity.
\end{proof}
\begin{rmk}
By \cite[Lemma~3.7]{ES}, a \nagychange{standard generator $S^{}_\alpha S^*_\beta$ with $\alpha\neq\beta$ is cycline} if and only if $(\alpha\gamma,\beta\gamma)$ are generalized cycles without entry in the sense of \cite{ES} for all $\gamma\in s(\alpha)\Lambda$.
\end{rmk}

\begin{ex}
\label{ex: 1 sn}
Let $E$ be a $1$-graph. We say a finite path $e_1\cdots e_n$ is a \emph{return path} if $r(e_1)=s(e_n)$; it \emph{has an entry} if there exists an $i$ and an edge $f\in r(e_i)E^1$ with $f\neq e_i$. We claim \nagychange{that a standard generator
$S^{}_\alpha S_\beta^*$ is cycline
 if and only if $\alpha=\beta$, $\alpha=\beta c$ or $\beta=\alpha c$ for some \jbchange{ $c\in s(\alpha)E s(\alpha)$ without entry.}} By \cite[Proposition-Definition~3.1]{NR12} $S^{}_\alpha S_\beta^*$ is normal if and only if $\alpha=\beta$, $\alpha=\beta c$ or $\beta=\alpha c$ for some  \jbchange{ $c\in s(\alpha)E s(\alpha)$ without entry.} Clearly $S^{}_\alpha S_\alpha^*$ is cycline; so it remains to show for $\alpha\neq \beta$.  Let $y$ be $\beta c c c\cdots$ if $\alpha=\beta c$  and $\alpha c c \cdots$ if $\beta=\alpha c$.  In either case $\{y\}=Z(\alpha)=Z(\beta)=Z(\alpha \gamma)=Z(\beta\gamma)$ for any $\gamma\in s(\alpha)E$.  Thus $P^{}_{\alpha\gamma}=P^{}_{\beta\gamma}$ for all $\gamma\in s(\alpha)E$ and so $S^{}_{\alpha} S_\beta^*$ is cycline.
\end{ex}

\begin{ex}
Let $E$ be a $1$-graph, $f: \N^k\to \N$ and $\Lambda=f^*E$, as in Example~\ref{ex: fE}.  Then a standard generator
$S^{}_\alpha S_\beta^*$ in $C^*(\Lambda)$ is cycline if and only if $p_1\alpha=p_1\beta$,  $p_1\alpha=p_1(\beta) c$, or $p_1\beta=p_1(\alpha) c$ for some  \jbchange{ $c\in s(\alpha)E s(\alpha)$ without entry.}  Indeed since $p_1 Z(\mu)=Z(p_1 \mu)$ we get $P^{}_{\alpha\gamma}=P^{}_{\beta\gamma}$ if and only if $P^{}_{p_1(\alpha\gamma)}= P^{}_{p_1(\beta\gamma)}$ and so the result follows from Example~\ref{ex: 1 sn}.
\end{ex}

\begin{rmk}
\begin{brown}
Kumjian and Pask construct for a row-finite $k$-graph $\Lambda$ with no sources an \'etale groupoid $G_\Lambda=\{(x,n-m, y)\in \Lambda^\infty \times \Z^k \times \Lambda^\infty \, | \, \sigma^n(x)=\sigma^m(y)\}$ and show $C^*(\Lambda)\cong C^*(G_\Lambda) $ \cite[Corollary~3.5]{KP00}.  A basis for a totally disconnected locally compact Hausdorff topology on $G_\Lambda$ is given by the cylinder sets
\[
Z(\alpha,\beta)=\{(\alpha x, d(\alpha)-d(\beta), \beta x) \, | \, x\in s(\alpha)\Lambda^\infty\}\quad \text{for $\alpha,\beta\in \Lambda$ with  $s(\alpha)=s(\beta)$.}
\]
The sets $Z(\alpha,\beta)$ are compact open. The isomorphism of Kumjian and Pask is characterized by $S^{}_\alpha S_\beta^*\mapsto \chi_{_{Z(\alpha,\beta)}}$ where $\chi_{_{Z(\alpha,\beta)}}$ is the characteristic function on $Z(\alpha,\beta)$.  Let $\Iso(G_\Lambda):=\{(x,p, y)\in G_\Lambda|x=y\}.$ Then $\interior(\Iso(G_\Lambda))$ is also an \'etale groupoid\footnote{Note $\Iso(G_\Lambda)$ typically is not \'etale} and $C^*(\interior(\Iso(G_\Lambda)))\subset C^*(G_\Lambda)$. Since $(\alpha,\beta)$ cycline if and only if $Z(\alpha,\beta)\subset \Iso(G_\Lambda)$, the isomorphism of $C^*(\Lambda)$ to $C^*(G_\Lambda)$ restricts to an isomorphism of $\nagychangeII{\mathscr{M}}$ with $C^*(\interior(\Iso(G_\Lambda)))$.  We suspect that we can use this groupoid formulation to show $\nagychangeII{\mathscr{M}}=\nagychangeII{\mathscr{M}'}$ but have yet to find a proof.
\end{brown}
\end{rmk}
%%%%%%%%%%%%%%%%%%%%%%%%%%%%%%%%%%%%%%%%%%%%%%%%%%%%%%%%%%%%%%%%%%%%%%%%%%%%%%%%%%%%%%%%%%%%%%%%%%%%%%%%%%%%%%%%%%%%%%%%%%%%%%%%%%%%%%%%%%%%%%%%%%%%%%%%%%%%%%%%%%%

\section{The ``aperiodic representation''}

Define $\Sigma=\{ (\alpha,\beta)\in \Lambda\times \Lambda | s(\alpha)=s(\beta), \alpha\neq \beta\}$.  Note that $\Sigma$ is countable since \sarchange{$\Lambda$ is.} For all $(\alpha,\beta)\in \Sigma$ let
\[
F_{\alpha,\beta}:=\{x\in \Lambda^\infty \, |\, x(0,d(\alpha))=\alpha, x(0,d(\beta))=\beta, \sigma^{d(\alpha)}(x)=\sigma^{d(\beta)}(x)\}
\]

Note that $F_{\alpha,\beta}=F_{\beta,\alpha}$.

\begin{rmk}
 $F_{\alpha,\beta}$ is closed for  all $(\alpha,\beta)\in \Sigma$.  Indeed, if $x_i\to x$ in $\Lambda^\infty$ and $n\in \N^k$ then $\sigma^{n}(x_i)\to \sigma^n(x)$.  In particular, if $x_i\in F_{\alpha,\beta}$ then $x(0,d(\alpha))=\alpha$,  $x(0, d(\beta))=\beta$ and $\sigma^{d(\alpha)}(x)=\sigma^{d(\beta)}(x)$ so that $x\in F_{\alpha,\beta}$.

As a consequence, $\partial F_{\alpha,\beta}\subset F_{\alpha,\beta}$. Note also that  $\partial F_{\alpha,\beta}$ is closed and meager.
\end{rmk}

\begin{defn}
\label{def: reg path}
For a $k$-graph $\Lambda$, we define the set of \emph{regular paths} to be

\[
\mathfrak{T}_\Lambda = \Lambda^\infty-\bigcup_{(\alpha,\beta)\in \Sigma}\partial F_{\alpha,\beta}.
\]
\end{defn}
We denote $\mathfrak{T}_\Lambda$ by $\mathfrak{T}$ when the graph is clear from context.

\begin{rmk}
\label{rmk: info on X}
Features of $\mathfrak{T}$
\begin{enumerate}
\item\label{it: X dense} $\mathfrak{T}$ is dense in $\Lambda^\infty$ by the Baire Category Theorem.  In particular for every $v\in \Lambda^0$ there exists an $x\in Z(v)\cap \mathfrak{T}$, that is, $v\mathfrak{T}\neq \emptyset$.
\item\label{it: X interior} $\mathfrak{T}\cap F_{\alpha,\beta}\subset \interior(F_{\alpha,\beta})$.
\item\label{it: X invariant 1} For $x\in \mathfrak{T}$ and $\nu\in \Lambda r(x)$ we have $\nu x\in \mathfrak{T}$.
\item\label{it: X invariant 2} For $x\in \mathfrak{T}$ and $n\in \N^k$ we have $\sigma^n(x)\in \mathfrak{T}$.
\end{enumerate}
For items \eqref{it: X invariant 1} and \eqref{it: X invariant 2} we use that $\nu Z(\mu)=Z(\nu\mu)$ and the equivalence $Z(\mu)\subset F_{\alpha,\beta}$ if and only if $Z(\nu\mu)\subset F_{\nu\alpha,\nu\beta}$.
\end{rmk}

\begin{ex}
\label{ex: 1graph}
 Suppose $E$ is a $1$-graph.  Then $\mathfrak{T}_E=E^\infty -\bigcup_{\Sigma} \partial F_{\alpha,\beta}$ is the set of aperiodic infinite paths along with those infinite paths that begin with a return path without entry: that is, $\mathfrak{T}_E$ consists precisely of the infinite {\em essentially aperiodic paths} described in \cite[Definition~2.5 (D)]{NR12}.

To see this, first suppose $x\in E^\infty$ is aperiodic.  Then $x\notin F_{\alpha,\beta}$ for any choice of $(\alpha,\beta)\in \Sigma$.  Since the $F_{\alpha,\beta}$ are closed,  $x\notin \partial F_{\alpha,\beta}$ and so aperiodic paths are in $\nagychangeII{\mathfrak{T}}$.  Next suppose that $x$ is periodic. Then there exists some $(\alpha,\beta)\in \Sigma$ such that $x\in F_{\alpha, \beta}$.  Without loss of generality we can assume $d(\beta)>d(\alpha)$ so that $\beta=\alpha c$ for some  $c\in s(\alpha)E s(\alpha)$.  So $x=\alpha c c c\cdots$.  In fact $F_{\alpha,\beta}=\{x\}$.  \sarchange{If $c$ has an entry, then for any $k \geq 1$, the set
  $Z(\alpha c^k)$ contains a  path of the form $z= \alpha c^k y$ where $y(0,d(c)) \neq c$.  We have
  $\sigma^{d(\alpha)}(z)((k-1)d(c),kd(c))=c \neq y(0,d(c))=\sigma^{d(\beta)}(z)((k-1)d(c),kd(c))$.  Thus $z\in   Z(\alpha c^k)\cap (\Lambda^\infty - F_{\alpha,\beta})$.  Since $k$ was arbitrary,   $x\in \partial F_{\alpha,\beta}$.} Thus periodic paths with entries are not in $\mathfrak{T}_E$.  If $c$ has no entry then $Z(\alpha c)=\{x\}= F_{\alpha,\beta}$ so that $F_{\alpha,\beta}$ is clopen and so $\partial F_{\alpha,\beta}=\emptyset$.  That is periodic paths without entries are in $\mathfrak{T}_E$.
\end{ex}

\begin{ex}
Let $\Lambda$ be as in Example~\ref{ex: fE}.  Recall the definition of the map $p_1: \Lambda^* \rightarrow E^*$ given in Claim~\ref{clm: well def}. Then
\[
\mathfrak{T}_\Lambda=\{x\in \Lambda^\infty: p_1x \text{~~ is  aperiodic or  \jbchange{begins with a return path without entry}}\}.
\]
This follows from the observation that $x\in F_{\alpha,\beta}$ if and only if $p_1 x\in F_{p_1 \alpha, p_1\beta}$ and Example~\ref{ex: 1graph} above.\footnote{Caution: $p_1\alpha$ can equal $p_1\beta$.}
\end{ex}

Consider $\mathfrak{T}$ as in Definition~\ref{def: reg path}.  Our goal is to define a representation of $C^*(\Lambda)$ on $\ell^2(\mathfrak{T})$. Let $\{\delta_x| x \in \mathfrak{T}\}$  be the canonical basis for $\ell^2(\mathfrak{T})$.  For each $\alpha\in \Lambda$ and $x\in \ell^2(\mathfrak{T})$ we put

\begin{equation}
\label{Talpha}
T_\alpha \delta_x=\begin{cases} \delta_{\alpha x} & \text{if } x\in s(\alpha)\Lambda^\infty\\
0 & \text{otherwise}.
\end{cases}
\end{equation}
Remark~\ref{rmk: info on X} shows that ${\alpha x}$ is indeed in $\mathfrak{T}$.     Note that
\[
T_\alpha^* \delta_x=\begin{cases} \delta_{\sigma^{d(\alpha)}(x)} & \text{if } x\in Z(\alpha)\\
0 & \text{otherwise}.
\end{cases}
\]
and $\alpha\mapsto T_\alpha$ gives a Cuntz-Kreiger $\Lambda$-family in $B(\ell^2(\mathfrak{T}))$.   Put $Q_\alpha=T_\alpha T_\alpha^*$.  By the universal property of $C^*(\Lambda)$ the correspondences $S^{}_\alpha\mapsto T_\alpha$ and $P_\alpha\mapsto Q_\alpha$ give a $*$-representation $\upsilon: C^*(\Lambda)\to B(\ell^2(\mathfrak{T}))$.

\jbchange{
\begin{defn}
\label{def flat rep}
For $\Lambda$ a $k$-graph and $\mathfrak{T}$ the set of regular paths in $\Lambda^\infty$.  The \emph{aperiodic representation} is the map $\upsilon: C^*(\Lambda)\to B(\ell^2(\mathfrak{T}))$ characterized by $S^{}_\alpha\mapsto T^{}_\alpha$ for all $\alpha\in \Lambda$.
\end{defn}
}

\begin{defn}
\label{defn: flat}
Call a cycline pair $(\alpha,\beta)$ \emph{special} if $T_\alpha=T_\beta$.  In this case $Q_\alpha=Q_\beta$.
\end{defn}

We use the following to characterize special cycline pairs.

\begin{lem}
\label{lem: compression}
For $(\alpha,\beta)\in \Sigma$ and $x\in \mathfrak{T}$ we have
\begin{enumerate}
\item \label{it: case x notin} If $x\notin F_{\alpha,\beta}$, then there exists $\mu,\nu\in \Lambda$ such that
\begin{align*}
& x\in Z(\mu)\cap Z(\nu), \quad \text{and}\\
& P_\mu S^{}_\alpha S_\beta^*P_\nu =0.
\end{align*}
\item \label{it: case x in} If $x\in F_{\alpha,\beta}$, then there exists $\gamma\in \Lambda$ with
\begin{align*}
& x\in Z(\alpha\gamma)\cap Z(\beta\gamma), \\
& P_{\alpha\gamma} S^{}_\alpha S_\beta^* P_{\beta\gamma}=S^{}_{\alpha\gamma}S_{\beta\gamma}^*,\quad\text{and}\\
& (\alpha\gamma,\beta\gamma) \text{~~special cycline}.
\end{align*}
\end{enumerate}
\end{lem}

\begin{proof}
 For \eqref{it: case x notin} first suppose that $x(0,d(\alpha))\neq \alpha$ or $x(0,d(\beta))\neq \beta$.  Then with $\mu=x(0,d(\alpha))$ and $\nu=x(0, d(\beta))$ we have $P_\mu S^{}_\alpha S_\beta^* P_\nu=0$.

Now suppose that $x(0,d(\alpha))=\alpha$ and $x(0,d(\beta))=\beta$.  The condition $x\notin F_{\alpha,\beta}$ means that
\[
\sigma^{d(\alpha)}(x)\neq \sigma^{d(\beta)} (x)
\]
so there exists an $n\in \N^k$ such that $\gamma:=x(d(\alpha),d(\alpha)+n)\neq x(d(\beta),d(\beta)+n)=:\eta$.  Put $\mu=\alpha\gamma$ and $\nu=\beta\eta$.  By definition $x\in Z(\mu)\cap Z(\nu)$ and
%by \eqref{proj form} (meaning the assumption? --- I don't think we need to mention it)
\[
P_\mu S^{}_\alpha S_\beta^* P_\nu=0
\]
as desired.

For \eqref{it: case x in}, assume $x\in F_{\alpha, \beta}$.  By Remark~\ref{rmk: info on X},  $\mathfrak{T}\cap F_{\alpha, \beta}=\interior F_{\alpha,\beta}$. Thus there exists $\epsilon\in \Lambda$ such that $x\in Z(\epsilon)\subset F_{\alpha,\beta}$.  Put
\begin{align*}
\alpha'&=x(d(\epsilon), d(\epsilon)+d(\alpha)),\\
\beta'&=x(d(\epsilon), d(\epsilon)+d(\beta)),
\end{align*}
and $\mu=\epsilon \alpha'$, $\nu=\epsilon\beta'$.  By definition $x\in Z(\mu)\cap Z(\nu)$.  Since $x$ is also in $Z(\alpha)\cap Z(\beta)$, by unique factorization we can write $\mu=\alpha\gamma$ and $\nu=\beta\gamma'$ where $\gamma=\sigma^{d(\alpha)}(x)(0, d(\epsilon))$ and $\gamma'=\sigma^{d(\beta)}(x)(0, d(\epsilon))$.  Since $x\in F_{\alpha,\beta}$, $\sigma^{d(\alpha)}(x)=\sigma^{d(\beta)}(x)$ and so $\gamma=\gamma'$.

Next we show $(\alpha\gamma, \beta\gamma)$ is cycline.   For this it suffices to show that $P_{\alpha\gamma\eta}=P_{\beta\gamma\eta}$ for all $\eta\in s(\gamma)\Lambda$ or equivalently $Z(\alpha\gamma\eta)=Z(\beta\gamma\eta)$.  Let $y\in Z(\alpha\gamma\eta)$ so that $y=\alpha\gamma \eta z=\epsilon \alpha' \eta z$ for some $z\in \Lambda^\infty$.  So $y\in Z(\epsilon)\subset F_{\alpha,\beta}$.  Thus there exists $z'\in \Lambda^\infty$ such that $y=\alpha z'=\beta z'$.  By unique factorization $z'=\gamma\eta z$ so we also have $y=\beta\gamma\eta z\in Z(\beta\gamma\eta)$.  That is $Z(\alpha\gamma\eta)\subset Z(\beta\gamma\eta)$ and by symmetry they are equal.

Finally we show that $(\alpha\gamma,\beta\gamma)$ is special.  It suffices to show
\[
T^*_{\alpha\gamma}\delta_y=T^*_{\beta\gamma} \delta_y \quad \forall ~y\in \mathfrak{T}.
\]
We already know that $Z(\alpha\gamma)=Z(\beta\gamma)$.   Thus if $y\notin Z(\alpha\gamma)$ then $T^*_{\alpha\gamma} \delta_y=0=T^*_{\beta\gamma} \delta_y$.
Now assume $y\in Z(\alpha\gamma)=Z(\beta\gamma)$.  Then since $\alpha\gamma=\epsilon \alpha'$ we have $y\in Z(\epsilon)$ and so $y\in F_{\alpha,\beta}$.  Thus $\sigma^{d(\alpha)}y=\sigma^{d(\beta)}y$ and so
\[
T_{\alpha}^* \delta_y=\delta_{\sigma^{d(\alpha)}y}=\delta_{\sigma^{d(\beta)}y}=T_\beta^* \delta_y.
\]
Therefore $T^*_{\alpha\gamma}\delta_y=T^*_{\beta\gamma} \delta_y$ as well.
\end{proof}

\section{States arising from the \nagychangeII{aperiodic} representation}

We begin by setting up notation:
\begin{align*}
A&:=\upsilon (C^*(\Lambda))\subset B(\ell^2(\mathfrak{T}))\\
D&:=\upsilon(\mathscr{D})=C^*(\{Q_\alpha\}_{\alpha\in \Lambda}\})\subset A\\
M&:=\upsilon(\mathscr{M})=\upsilon(C^*(\{S^{}_\alpha S_\beta^*\, | \, (\alpha,\beta)  \text{ cycline}\})).
\end{align*}

Notice that
\begin{itemize}
\item $ \upsilon(\sarchange{\mathscr{D}'})\subset A\cap D'$
\item $\upsilon|_{\mathscr{D}}$ is a $*$-isomorphism from $\mathscr{D}$ to $D$
\end{itemize}

\begin{rmk}
\label{rmk: ev}
Every $x\in \mathfrak{T}$ determines a pure state $\ev_x^{\mathscr{D}}$ on $\mathscr{D}$ and thus also a pure state $\ev_x^D$ on $D$.
\end{rmk}

\begin{rmk}
\label{rmk: abelian}
For any $x\in \mathfrak{T}$ the net $\{Q_{x(0,n)}\}_{n\in\N^k}\subset B(\ell^2(\mathfrak{T}))$  converges in the strong operator topology to the orthogonal projection $p_x$ onto $\C \delta_x$.  In particular, if $T\in B(\ell^2(\mathfrak{T}))$ commutes with $D$, then $T$ commutes with the all of the projections $p_x$ and so $T\in \ell^\infty (\mathfrak{T})$.  In other words, the commutant $D'$ of $D$ in $B(\ell^2(\mathfrak{T}))$ is simply $\ell^\infty(\mathfrak{T})$.  It then follows that both $M$ and $D'\cap A$ are abelian since both are subalgebras of the abelian $C^*$-algebra $D'$.
\end{rmk}

\begin{prop}
\label{prop: ev unique ext}
For each $x\in \mathfrak{T}$ the pure state $\ev_x^D$ on $D$ has a unique extension to a pure state $\phi_x$ on $A$.
\sarchange{In particular,
\begin{equation} \label{ev ext fmla alpha not= beta}
\phi(T_\alpha T_\beta^*)=\begin{cases} 1 & \text{if~} x\in F_{\alpha,\beta}\\
0 & \text{if ~} x\notin F_{\alpha,\beta}. \end{cases}
\end{equation}}

% Changed to the above because the statement mentioned $\mu$ and $\nu$ without defining or quantifying over them, and I could not find a reference to this equation later.

%\jbchange{In particular,
%\begin{equation} \label{ev ext fmla alpha not= beta}
%\phi(T_\alpha T_\beta^*)=\phi(Q_\mu T_\alpha T_\beta^* Q_\nu)=\begin{cases} 1 & \text{if~} x\in F_{\alpha,\beta}\\
%0 & \text{if ~} x\notin F_{\alpha,\beta}. \end{cases}
%\end{equation}}
\end{prop}

\begin{proof}
Fix some state $\phi$ on $A$ with $\phi|_D=\ev_x^D$.  We show the values $\phi(T_\alpha T_\beta^*)$ for $\alpha,\beta\in \Lambda$ with $s(\alpha)=s(\beta)$ depend only on $\alpha,\beta$ and $x$.

First note that if $\alpha=\beta$ then
\begin{equation} \label{ev ext fmla alpha=beta}
\phi(T_\alpha T_\beta^*)=\phi(Q_\alpha)=\ev_x^D(Q_\alpha)=\begin{cases} 1 & \text{if~} x\in Z(\alpha)\\ 0 & \text{otherwise}.\end{cases}
\end{equation}

Now suppose $\alpha\neq \beta$, that is $(\alpha,\beta)\in \Sigma$.  \jbchange{Equation~\eqref{ev ext fmla alpha=beta} shows $\phi(Q_\mu)=1=\phi(Q_\nu)$ whenever $x\in Z(\mu)\cap Z(\nu)$ so in this case Proposition~\ref{state eq} gives that
$\phi(T_\alpha T_\beta^*)=\phi(Q_\mu T_\alpha T_\beta^* Q_\nu)$. }Now using Lemma~\ref{lem: compression} it follows that  there exists $\mu,\nu$ so that $x\in Z(\mu)\cap Z(\nu)$ with
\[
\phi(T_\alpha T_\beta^*)=\phi(Q_\mu T_\alpha T_\beta^* Q_\nu)=\begin{cases} 1 & \text{if~} x\in F_{\alpha,\beta}\\
0 & \text{if ~} x\notin F_{\alpha,\beta}. \end{cases}\qedhere
\]
\end{proof}

\begin{rmk}
\label{rmk: states distinct}
The pure states $\ev_x^D$ on $D$ are \emph{distinct}.  Indeed, if $x_1\neq x_2$ then there exists an $n\in \N^k$ such that $x_1(0,n)\neq x_2(0,n)$.  Thus for $i,j\in \{1,2\}$, the projections $P_{x_i(0,n)}$ are orthogonal and so $\ev_{x_i}^D(Q_{x_j(0,n)})=\delta_{i,j}.$
\end{rmk}

\begin{rmk}
\label{rmk: state on M}
The state $\ev^M_x:=\phi_x|_M$ is pure.
This follows since $\phi_x(T)=\langle T \delta_x|\delta_x\rangle$ so in particular this equation holds for all $T\in \ell^\infty (\mathfrak{T})$ (and thus $\ev^M_x$ is a $*$-ho\-mo\-mor\-phism).
\end{rmk}

\begin{brown}
\begin{rmk}
\label{rmk: jf}
The states $\{\phi_x\}_{x\in \mathfrak{T}}$ are jointly faithful on $A$ (\nagychangeII{thus by restriction they are faithful on both $D$ and $M$}).  Indeed,  for a positive operator $T\in A\in B(\ell^2(\mathfrak{T}))$ we have $\phi_x(T)=\langle T\delta_x|\delta_x\rangle$ so that if $\phi_x(T)=0$ for all $x\in \nagychangeII{\mathfrak{T}}$,  $T^{1/2}\delta_x=0$ for all $x\in \nagychangeII{\mathfrak{T}}$.  This gives $T^{1/2}=0$ and so $T=0$.
%\footnote{J: made a remark because we use this prominent later.}
\end{rmk}
\end{brown}

\begin{rmk}
\label{rmk: X dense in M}
Since $M$ is abelian, $M=C_0(\Omega)$ for some locally compact Hausdorff space $\Omega$.  The correspondence $x\mapsto \ev_x^M$ yields an inclusion  $\nagychangeII{\mathfrak{T}}\subset \Omega$.  As observed above, $\{\ev_x^M\}_{\nagychangeII{x\in\mathfrak{T}}}$ is jointly faithful on $M$ and so $\nagychangeII{\mathfrak{T}}$ is dense in $\Omega$.
\end{rmk}

%%%%%%%%%%%%%%%%%%%%%%%%%%%%%%%%%%%%%%%%%%%%%%%%%%%%%%%%%%%%%%%%%%%%%%%%%%%%%%%%%%%%%%%%%%%%%%%%%%%%%%%%%%%%%%%%%%%%%%%%%%%%%%%%%%%%%%%%%%%%%%%%%%%%%%%%%%%%%%%%%%%

\section{The Twisted aperiodic Representation}
\jbchange{In this section we augment the aperiodic representation to get an injective representation of $C^*(\Lambda)$.  We then use this representation to prove our main theorem.}

\begin{defn}
\label{defn twist rep}
The \emph{twisted aperiodic representation} of $C^*(\Lambda)$ is the $*$-homomor\-phism
\begin{align*}
\Upsilon: C^*(\Lambda)&\to C(\T^k)\otimes B(\ell^2(\mathfrak{T}))\quad \text{characterized by}\\
  S^{}_\alpha &\mapsto h_{d(\alpha)}\otimes T_\alpha\\
  P_\alpha & \mapsto 1\otimes Q_\alpha.
  \end{align*}
\end{defn}
There is an obvious gauge action on $C(\T^k)\otimes B(\ell^2(\mathfrak{T}))$ and $\Upsilon$ is equivariant with respect to the gauge actions.  Since $\Upsilon|_\mathscr{D}$ is injective the gauge invariant uniqueness theorem \cite[Theorem~4.1]{RSY03} gives that $\Upsilon$ is injective.

\begin{rmk}
\label{rmk: ab}
By construction $\Upsilon(C^*(\Lambda))\subset C(\T^k)\otimes A$ and $\Upsilon(\mathscr{M})\subset C(\T^k)\otimes M$  so that $\Upsilon(\mathscr{M})$ (thus also $\mathscr{M}$) is abelian.

\end{rmk}
\jbchange{As a quick consequence of the isomorphism $\Upsilon$ we are able to describe $\mathscr{D}'$ more fully; we show that $\mathscr{D}'\subset C^*(\Lambda)$ is also abelian and that in fact $\mathscr{D}'=\mathscr{M}'$.}

\begin{prop}
\label{prop: com}
\jbchange{Let $\Lambda$ be a row-finite $k$-graph with no sources, $\mathscr{D}=C^*(\{P_\alpha\})$, $\mathscr{M}=C^*(\{S^{}_\alpha S_\beta^*\, | \, (\alpha,\beta)  \text{ cycline}\})$.  Then}
\begin{enumerate}
\item\label{it: b ab}$\mathscr{D}'$ is abelian.
\item\label{it: b m}  $\mathscr{D}'=\mathscr{M}'$.
\end{enumerate}
\end{prop}

\begin{proof}
Let $b\in \mathscr{D}'$, then $\Upsilon(b)\in C(\T^k)\otimes A$ commutes with  $\Upsilon(P_\alpha)$ for all $\alpha\in \Lambda$.  For $z\in \T^k$ let $\varepsilon_z: C(\T^k)\otimes A\to A$ be the evaluation map.  Then the element $\varepsilon_z(\Upsilon(b))\in A$ commutes with all $\varepsilon_z(\Upsilon(P_\alpha))=Q_\alpha$.  So $\varepsilon_z(\Upsilon(b))\in D'=\ell^\infty(\mathfrak{T})$.  In other words
\[
\Upsilon(b)\in C(\T^k)\otimes \ell^\infty(\mathfrak{T}).
\]
Firstly, $\Upsilon(\mathscr{D}')$ is a subalgebra of an abelian $C^*$-algebra and hence abelian;  since $\Upsilon$ is an isomorphism, $\mathscr{D}'$ is abelian too giving \eqref{it: b ab}.  Secondly, since $M\subset \ell^\infty(\mathfrak{T})$ we get that $\Upsilon(b)$ commutes with $C(\T^k)\otimes M$.  In particular $\Upsilon(b)$ commutes with $\Upsilon(\mathscr{M})$. Hence $\mathcal{D}'\subset \mathscr{M}'\subset \mathscr{D}'$, giving \eqref{it: b m}.
\end{proof}

\jbchange{Recall that for  $n\in \Z^k$, we denote the associated character on $\T^k$ by $h_n$, that is $h_n(z_1,...,z_k)=z_1^{n_1}\cdots z_k^{n_k}$. }

\begin{lem}
\label{lem: H sets equal}
For  any $x\in \mathfrak{T}$,   the following sets of functions in $C(\T^k)$ are all equal.
\begin{align*}
H_x &:= \{1\}\cup\{ h_{d(\alpha)-d(\beta)}\, |\, (\alpha,\beta)\in \Sigma, \,  x\in F_{\alpha,\beta}\},\\
H^c_x &:= \{1\}\cup\{ h_{d(\alpha)-d(\beta)}\, |\, (\alpha,\beta)\in \Sigma\,\text{ with $(\alpha,\beta)$ cycline, } x\in F_{\alpha,\beta}\},\\
H^s_x &:= \{1\}\cup\{ h_{d(\alpha)-d(\beta)} \, |\, (\alpha,\beta)\in \Sigma\text{ with $(\alpha,\beta)$ special cycline, } x\in F_{\alpha,\beta}\}.
\end{align*}
\end{lem}
\begin{brown}
\begin{rmk}
In \cite{CKSS} \nagychangeII{Carlsen, Kang, Shotwell and Sims define
$$\text{Per}(\Lambda):=\{d(\alpha)-d(\beta)\, |\, (\alpha,\beta) \text{~cycline}\}.$$}They show that if for every $v_1,v_2\in \Lambda^0$ there exists $\lambda_1, \lambda_2\in \Lambda$ such that $s(\lambda_1)=s(\lambda_2)$ and $r(\lambda_i)=v_i$ $i\in \{1,2\}$ then $\text{Per}(\Lambda)$ is a group \cite[Theorem~4.2]{CKSS}.  The same proof shows that in this case $H^c_x\cong\text{Per}(\Lambda)$ for every $x\in \mathfrak{T}$  and thus they are all isomorphic as groups.
\end{rmk}
\end{brown}
\begin{proof}[Proof of Lemma~\ref{lem: H sets equal}:]
By definition $H^s_x\subset H^c_x\subset H_x$, \jbchange{so it suffices to show $H_x\subset H^s_x$}.  If $g\in H_x-\{1\}$, then $g=h_{d(\alpha)-d(\beta)}$ for some $(\alpha,\beta)\in \Sigma$ with $x\in F_{\alpha,\beta}$.  By Lemma~\ref{lem: compression} \eqref{it: case x in}, there exists $\gamma$ such that $(\alpha\gamma,\beta\gamma)$ special cycline, $x\in Z(\alpha\gamma)\cap Z(\beta\gamma)$ and
\[
P_\mu S^{}_\alpha S_\beta^* P_\nu=S^{}_{\alpha\gamma} S_{\beta\gamma}.
\]
We still have $x\in F_{\alpha\gamma,\beta\gamma}$ and $d(\alpha\gamma)-d(\beta\gamma)=d(\alpha)-d(\beta)$ so $g=h_{d(\alpha\gamma)-d(\beta\gamma)}$ proving $g\in H_x^{\nagychangeII{s}}$.
\end{proof}

For every $(z,x)\in \T^k\times \mathfrak{T}$, by Proposition~\ref{pure 1} we have that $\ev_z\otimes \phi_x$ is a pure state on $C(\T^k)\otimes A$. Consider the following states
\begin{align*}
\psi_{z,x}&= \ev_z\otimes \phi_x|_{\Upsilon(C^*(\Lambda))}\\
e_{z,x}&=\psi_{z,x}|_{\Upsilon(\mathscr{M})}.
\end{align*}
Since $\Upsilon(\mathscr{M})\subset C(\T^k)\otimes M\subset C(\T^k)\otimes A$ we can also write $e_{z,x}=(\ev_z\otimes \ev_x^M)|_{\Upsilon(\mathscr{M})}.$

\begin{rmk}
\label{rmk: maps}
We should caution the reader that the maps from $\T^k\times \mathfrak{T}$ to states on $\Upsilon(C^*(\Lambda))$ and $\Upsilon(\mathscr{M})$ given by
\begin{align*}
(z,x) &\mapsto \psi_{z,x}\\
(z,x)&\mapsto e_{z,x}
\end{align*}
are not injective in general (See Lemma~\ref{lem:main}).
\end{rmk}

Let   $\mathcal{Y}$ be the collection of states on $\Upsilon(C^*(\Lambda))$ and $\mathcal{X}$ be the collection of pure states on $\Upsilon(\mathscr{M})$  given by
\begin{align}
\mathcal{Y}&=\{\psi_{z,x}| (z,x)\in \T^k\times \nagychangeII{\mathfrak{T}}\}\label{special state}\\
\mathcal{X}&=\{e_{z,x}| (z,x)\in \T^k\times \nagychangeII{\mathfrak{T}}\}.\nonumber
\end{align}

\begin{lem}\label{lem:jointfaith}
\begin{enumerate}
%\hspace*{-20pt}
\item\label{it:Y} \nagychangeII{
$\mathcal{Y}$ is a jointly faithful set of states on $\Upsilon(C^*(\Lambda))$.}
\item\label{it:X} $\mathcal{X}$ is a jointly faithful set of states on $\Upsilon(\mathscr{M})$.
\end{enumerate}
\end{lem}

\begin{proof}
Since the restriction map $\psi\mapsto \psi|_{\Upsilon(\mathscr{M})}$ maps $\mathcal{Y}$ onto $\mathcal{X}$ it suffices to prove~\eqref{it:Y}.  But this is trivial since $\{\ev_z\otimes\phi_x\}$ is a jointly faithful collection of states on $C(\T^k)\otimes A$ (see Remark~\ref{rmk: jf}).
\end{proof}

Before getting to our next theorem we need a lemma that relates the kernels of the maps in Remark~\ref{rmk: maps}.

\begin{lem}
\label{lem:main}
For $z_1, z_2\in \T^k$ and $x\in \nagychangeII{\mathfrak{T}}$, the following are equivalent:
\begin{enumerate}
\item\label{it:e} $e_{z_1, x}=e_{z_2,x}$;
\item\label{it:psi}$\psi_{z_1,x}=\psi_{z_2,x}$;
\item\label{it:h}$h(z_1)=h(z_2)$ for all $h\in H_x$.
\end{enumerate}
\end{lem}

\begin{proof}
For $\eqref{it:e}\Rightarrow \eqref{it:h}$, \jbchange{suppose $e_{z_1, x}=e_{z_2,x}$} and let $h\in H_x$ be given.  Since $H_x=H_x^s$ there exists $(\alpha,\beta)\in \Sigma$  cycline with $x\in F_{\alpha,\beta}$ and $h=h_{d(\alpha)-d(\beta)}$.   Then $S^{}_\alpha S_\beta^*\in \mathscr{M}$ and so $\Upsilon(S^{}_\alpha S_\beta^*)=h\otimes T_\alpha T_\beta^*$ with  $T_\alpha T_\beta^*\in M$.  Now
\begin{align*}
h(z_1)=h(z_1)\phi_x(T_\alpha T_\beta^*)&=e_{z_1, x}(\Upsilon (S^{}_\alpha S_\beta^*))\\
&=e_{z_2,x}(\Upsilon (S^{}_\alpha S_{\beta}^*))=h(z_2)\phi_x(T_\alpha T_\beta^*)=h(z_2)
\end{align*}
giving \eqref{it:h}.

For $\eqref{it:h}\Rightarrow \eqref{it:psi}$, suppose $h(z_1)=h(z_2)$ for all $h\in H_x$.  Since
\nagychangeII{$$\psi_{z,x}=\ev_x\otimes \phi_x|_{(\Upsilon(C^*(\Lambda)))},$$by Proposition~\ref{prop: ev unique ext}  it follows} that
\[
\psi_{z_i,x} (\Upsilon(S^{}_\alpha S_\beta^*))=\begin{cases} 1 & \text{if~} \alpha=\beta \quad\text{and}\quad x\in Z(\alpha)\\
0 & \text{if~} \alpha=\beta \quad\text{and}\quad x\notin Z(\alpha)\\
h_{d(\alpha)-d(\beta)}(z_i) & \text{if~} (\alpha,\beta)\in \Sigma\quad \text{and}\quad x\in F_{\alpha,\beta}\\
0 & \text{if~} (\alpha,\beta)\in \Sigma\quad \text{and}\quad x\notin F_{\alpha,\beta}.
\end{cases}
\]
\jbchange{But by  \eqref{it:h} we have $h_{d(\alpha)-d(\beta)}(z_1)=h_{d(\alpha)-d(\beta)}(z_2)$, therefore $\psi_{z_1,x} (\Upsilon(S^{}_\alpha S_\beta^*))=\psi_{z_2,x} (\Upsilon(S^{}_\alpha S_\beta^*))$}.

Since $e_{z_i,x}=\psi_{z_i,x}|_{\Upsilon(\mathscr{M})}$, $\eqref{it:psi}\implies \eqref{it:e}$ is trivial.
\end{proof}

\begin{thm}
\label{thm: state ext}
Let $\Lambda$ be a row-finite $k$-graph with no sources, $\mathfrak{T}$ be the set of regular paths in $\Lambda^\infty$ and $\mathcal{X}$  be  as in \eqref{special state}. Then
every state in $\mathcal{X}$ extends uniquely to a state on $\Upsilon(C^*(\Lambda))$.  Specifically, if $\rho\in \mathcal{X}$ is presented as $\rho=e_{z,x}$ for some pair $(z,x)\in \T^k\times \mathfrak{T}$, then its unique extension is $\psi_{z,x}$.\footnote{Caution:  $\rho$ may be presented in many different ways!}
\end{thm}

\begin{proof}
Fix $\rho=e_{z_0,x_0}$ for some $(z_0,x_0)\in \T^k\times \mathfrak{T}$.  By way of contradiction, assume there are two distinct pure states $\xi_1,\xi_2$ on $\Upsilon(C^*(\Lambda))$ such that $\xi_1|_{\Upsilon(\mathscr{M})}=\xi_2|_{\Upsilon(\mathscr{M})}=\rho$.  For each $i\in \{1,2\}$ extend $\xi_i$ to a pure state $\zeta_i$ on $C(\T^k)\otimes A$.  Observe that when restricted to $\Upsilon(\mathscr{D})=\C\otimes D\subset \Upsilon(\mathscr{M})\subset\Upsilon(C^*(\Lambda))$ we have for all $\alpha\in \Lambda$

\[
\zeta_1(1\otimes Q_\alpha)=\xi_1(1\otimes Q_\alpha)=\rho(1\otimes Q_\alpha)=\xi_2(1\otimes Q_\alpha)=\zeta_2(1\otimes Q_\alpha).
\]
\jbchange{By Proposition~\ref{prop: ev unique ext}, $\ev_{x_0}^D$ extends uniquely to $\phi_{x_0}$. Hence Proposition~\ref{pure 2} shows there exist} $z_1, z_2\in \T^k$ such that
\[
\zeta_i=\ev_{z_i}\otimes \phi_{x_0}.
\]
By restricting to $\Upsilon(C^*(\Lambda))$ we have
\[
\xi_i=(\ev_{z_i}\otimes \phi_{x_0})|_{\Upsilon(C^*(\Lambda))}=\psi_{z_i,x_0}.
\]
So now we have $z_1, z_2\in \T^k$ such that $\psi_{z_1, x_0}$ and $\psi_{z_2,x_0}$ are distinct states on $\Upsilon(C^*(\Lambda))$ but $e_{z_1,x_0}=e_{z_2,x_0}=\rho$. This contradicts Lemma~\ref{lem:main}.
\end{proof}

\jbchange{Theorem~\ref{thm: state ext} along with Proposition~\ref{prop: main state} now conspire to obtain a uniqueness theorem for higher-rank graphs.}

\begin{thm}
\label{thm: inj}
\jbchange{Let $\Lambda$ be a row-finite $k$-graph with no sources. If $B$ is a $C^*$-algebra and $\pi: C^*(\Lambda)\to B$ is a $*$-ho\-mo\-mor\-phism that is injective \nagychangeII{
when restricted to}
$\mathscr{M}=C^*(\{S^{}_\alpha S_\beta^* \,  |\, (\alpha,\beta)\text{\rm\ cycline}\})$, then $\pi$ is injective.}
\end{thm}

\begin{proof}
By Remark~\ref{rmk: ab} we have $\mathscr{D}\subset \mathscr{M}$ are abelian subalgebras of $C^*(\Lambda)$. Since $\Upsilon: C^*(\Lambda)\to \Upsilon(C^*(\Lambda))$ is an isomorphism we can  pull back the pure states $\mathcal{X}$ on $\Upsilon(\mathscr{M})$ to get a set of pure states $\mathscr{X}$ on $\mathscr{M}$. By Theorem~\ref{thm: state ext} each of the pure states in $\mathcal{X}$ has unique extension so the same is true for the states in $\mathscr{X}$.     Further, by Lemma~\ref{lem:jointfaith} we know $\mathcal{X}$ is jointly faithful on $\Upsilon(\mathscr{M})$ so the set $\mathscr{X}$ is jointly faithful on $\mathscr{M}$.  Thus  Theorem~\ref{prop: main state} now gives that $\pi$ is injective.
\end{proof}

\begin{rmk}
Since $\mathscr{M}=\mathscr{D}$ if $\Lambda$ is aperiodic (see Remark~\ref{rmk: ap}), Theorem~\ref{thm: inj} recovers the usual Cuntz-Krieger uniqueness theorem for row-finite higher-rank graph algebras \cite[Theorem~4.6]{KP00}.
\end{rmk}

%%%%%%%%%%%%%%%%%%%%%%%%%%%%%%%%%%%%%%%%%%%%%%%%%%%%%%%%%%%%%%%%%%%%%%%%%%%%%%%%%%%%%%%%%%%%%%%%%%%%%%%%%%%%%%%%%%%%%%%%%%%%%%%%%%%%%%%%%%%%%%%%%%%%%%%%%%%%%%%%%%%

%%%%%%%%%%%%%%%%%%%%%%%%%%%%%%%%%%%%%%%%%%%%%%%%%%%%%%%%%%%%%%%%%%%%%%%%%%%%%%%%%%%%%%%%%%%%%%%%%%%%%%%%%%%%%%%%%%%%%%%%%%%%%%%%%%%%%%%%%%%%%%%%%%%%%%%%%%%%%%%%%%%

\end{document}